\documentclass[a4paper,fleqn]{cas-sc}

\usepackage[numbers]{natbib}
\usepackage{amsmath,amssymb,amsfonts,amsthm}
\usepackage[textsize=scriptsize]{todonotes}
\usepackage{cleveref}
\usepackage{algorithm}
\usepackage{algpseudocode}

\DeclareMathOperator{\diag}{diag}

\DeclareMathOperator{\tr}{tr}

\DeclareMathOperator{\vect}{vec}
\DeclareMathOperator{\orth}{orth}

\newcommand{\R}{\mathbb{R}}

\newcommand{\sH}{\mathcal{H}}
\newcommand{\norm}[1]{\left\lVert#1\right\rVert}
\newcommand{\abs}[1]{\left\lvert#1\right\rvert}

\newtheorem{theorem}{Theorem}[section]
\newtheorem{lemma}[theorem]{Lemma}
\newtheorem{proposition}[theorem]{Proposition}
\newtheorem{corollary}[theorem]{Corollary}
\newtheorem{definition}[theorem]{Definition}
\newtheorem{remark}[theorem]{Remark}

\begin{document}
\let\WriteBookmarks\relax
\def\floatpagepagefraction{1}
\def\textpagefraction{.001}
\shorttitle{Truncated $\mathcal{H}_2$ Optimal Model Reduction for BQO Systems}
\shortauthors{H. Faßbender et~al.}
%\begin{frontmatter}

\title [mode = title]{Truncated $\mathcal{H}_2$ Optimal Model Reduction for Bilinear Quadratic Output Systems}                      
% \tnotemark[1,2]

% \tnotetext[1]{This document is the results of the research
%    project funded by the National Science Foundation.}

% \tnotetext[2]{The second title footnote which is a longer text matter
%    to fill through the whole text width and overflow into
%    another line in the footnotes area of the first page.}

\author[1]{Heike Faßbender}[%type=editor,
                        %auid=000,bioid=1,
                        %role=Researcher,
                        orcid=0000-0002-0970-2821]
%\cormark[2]
\ead{h.fassbender@tu-braunschweig.de}
\ead[url]{https://www.tu-braunschweig.de/inum/personal/fassbender}

\credit{Writing - Original Draft, Conceptualization, Supervision} %Writing - Review\& Editing

\affiliation[1]{organization={Institute for Numerical Analysis, TU Braunschweig},
                addressline={Universitätsplatz 2}, 
                %citysep={},
                %postcode={38106},
                city={38106 Braunschweig},
              % citysep={}, % Uncomment if no comma needed between city and postcode
                state={Lower Saxony},
                country={Germany}}

\author[2]{Serkan Gugercin}[orcid=0000-0003-4564-5999]
%\cormark[1]
\ead{gugercin@vt.edu}
\ead[url]{https://gugercin.math.vt.edu}

\affiliation[2]{organization={Department of Mathematics and Division of Computational Modeling and Data Analytics, Academy of Data Science, Virginia Tech},
                addressline={225 Stanger Street, 460 McBryde Hall}, 
                %citysep={},                
                city={Blacksburg},
                postcode={24061},
              % citysep={}, % Uncomment if no comma needed between city and postcode
                state={Virginia},
                country={United States of America}}

\credit{Writing - Original Draft, Conceptualization, Methodology}
%Writing - Review\& Editing

\author[1]{Till Peters}[orcid=0009-0009-2852-6913]
\cormark[1]
\ead{till.peters@tu-braunschweig.de}
\ead[URL]{https://www.tu-braunschweig.de/inum/personal/till-peters}

\credit{Conceptualization, Methodology, Software, Investigation, Writing - Original Draft} %Writing - Review\& Editing

\cortext[cor1]{Corresponding author}
%\cortext[cor2]{Principal corresponding author}
% \fntext[fn1]{This is the first author footnote, but is common to third
% \fntext[fn2]{Another author footnote, this is a very long footnote and
%   it should be a really long footnote. But this footnote is not yet
%   sufficiently long enough to make two lines of footnote text.}

%\nonumnote{This note has no numbers. In this work we demonstrate $a_b$
  % the formation Y\_1 of a new type of polariton on the interface
  % between a cuprous oxide slab and a polystyrene micro-sphere placed
  % on the slab.}

\begin{abstract}
We develop an $\mathcal{H}_2$ model reduction framework for polynomial bilinear systems with quadratic outputs. We define the associated $\mathcal{H}_2$ norm and show its consistency with the established norm in the original (non-polynomial) BQO case. We derive first-order necessary conditions for $\mathcal{H}_2$ optimal model reduction of these polynomial systems and use them to construct an efficient iterative algorithm for computing reduced BQO models that approximately satisfy these conditions. Numerical results on two examples demonstrate the accuracy and effectiveness of the proposed reduction approach.
\end{abstract}

% \includegraphics{figs/cas-grabs.pdf}

% \begin{highlights}
% \item Research highlights item 1
% \item Research highlights item 2
% \item Research highlights item 3
% \end{highlights}


\maketitle

\section{Introduction}
\label{sec:1}

We investigate the bilinear quadratic output (BQO) dynamical systems represented in the state-space form as
\begin{subequations}\label{eq:BQO}
    \begin{align}
    \dot{x}(t) &= Ax(t) + \sum_{k=1}^m N_k x(t) u_k(t)+ Bu(t),\, x(0)=0, \label{eq:BQO_a1}\\
    y(t)&= Cx(t)+\begin{bmatrix}
        x(t)^TM_1x(t) \\ \vdots \\ x(t)^TM_px(t)
    \end{bmatrix}, \label{eq:BQO_a2}
\end{align}
\end{subequations}
where $A\in\R^{n\times n}, B\in\R^{n \times m}, C\in \R^{p\times n}, N_k\in\R^{n\times n}$ for $k=1,\dots, m$, $M_j\in \R^{n\times n}$ for $j=1,\dots, p$, and $t\in T:=[0,\infty)$. To make the notation concise, we write $\mathcal{S}=(A,B,C,\mathcal{N},\mathcal{M})$ for the BQO system described in \eqref{eq:BQO} where we define
    \begin{align}
    \mathcal{N}:=\begin{bmatrix}
        N_1 \,  \dots \, N_m
    \end{bmatrix}\in \R^{n\times nm}, \quad 
    \mathcal{N}_T:=\begin{bmatrix}
        N_1^T \,  \dots \, N_m^T
    \end{bmatrix}\in \R^{n\times nm}, \quad
    \mathcal{M}:=\begin{bmatrix}
        M_1 \,  \dots \, M_p
    \end{bmatrix}\in \R^{n\times np}. \label{eq:mathcal}
\end{align} Here, $x(t)\in\R^n$ denotes  the state of the system, $u(t)\in \R^m$ denotes the input, and $y(t)\in \R^p$ denotes the output. We can assume without loss of generality that $M_j$ is symmetric for all $j=1,\dots, p$ due to $x(t)^TM_jx(t)=x(t)^TM_j^Tx(t)=\frac{1}{2}x(t)^T(M_j+M_j^T)x(t)$.

In principal, BQO systems can be considered as a combination of bilinear systems (see, e.g., \cite{al-baiyat_new_1993,zhang_2002}) and linear quadratic output (LQO) systems (see, e.g., \cite{benner_lqo_2022}). By setting $M_1=\dots=M_p=0$ in \eqref{eq:BQO_a2}, we recover bilinear systems, and similarly, setting $N_1=\dots=N_m=0$ in \eqref{eq:BQO_a1} yields the LQO system equations.

In most cases, such as those where the dynamical system arises from the spatial discretization of an underlying PDE, the dimension of the resulting BQO system is significantly large, and the task of model order reduction arises.
Thus, we seek a BQO system of same structure but with a significantly lower dimension that approximates the input-output behavior well in an appropriate norm. More precisely, we look for the reduced BQO system 
of order $r \ll n$ described as
\begin{subequations}\label{eq:BQOr}
    \begin{align}
    \dot{\widehat{x}}(t) &= \widehat{A}\widehat{x}(t) + \sum_{k=1}^m \widehat{N}_k \widehat{x}(t) u_k(t)+ \widehat{B}u(t), \, \widehat{x}(0)=0, \label{eq:BQOr_a1}\\
    \widehat{y}(t)&=\widehat{C}\widehat{x}(t)+\begin{bmatrix}
        \widehat{x}(t)^T\widehat{M}_1\widehat{x}(t) \\ \vdots \\ \widehat{x}(t)^T\widehat{M}_p\widehat{x}(t)
    \end{bmatrix}, \label{eq:BQOr_a2}
\end{align}
\end{subequations}
where $\widehat{A}\in\R^{r\times r}, \widehat{B}\in\R^{r \times m}, \widehat{C}\in \R^{p\times r}$, $\widehat{N}_k\in\R^{r\times r}$ for $k=1,\dots, m$, $\widehat{M}_j\in \R^{r\times r}$ for $j=1,\dots, p$, $t\in T$. Again, we write in shorthand $\widehat{\mathcal{S}}=(\widehat{A},\widehat{B},\widehat{C},\widehat{\mathcal{N}},\widehat{\mathcal{M}})$ for the BQO system $\widehat{\mathcal{S}}$, using $\widehat{\mathcal{N}}:=\begin{bmatrix}
        \widehat{N}_1 \,  \dots \, \widehat{N}_m
    \end{bmatrix}\in \R^{r\times rm}$,    $\widehat{\mathcal{N}}_T:=\begin{bmatrix}
        \widehat{N}_1^T \,  \dots \, \widehat{N}_m^T
    \end{bmatrix}\in \R^{r\times rm}$,
    $\widehat{\mathcal{M}}:=\begin{bmatrix}
        \widehat{M}_1 \,  \dots \, \widehat{M}_p
    \end{bmatrix}\in \R^{r\times rp}$. In \eqref{eq:BQOr}, $\widehat{x}(t)\in\R^r$ describes the reduced state and $\widehat{y}(t)\in \R^p$ the reduced output of the system. The goal is to construct \eqref{eq:BQOr} such that
the reduced output $\widehat{y}(t)$ provides a high-fidelity approximation to the original output $y(t)$.

 The most common way to generate such a reduced system is via Petrov-Galerkin projection. To this end, we assume that the state $x(t)$ lies approximately in a low dimensional subspace spanned by the columns of a matrix $V\in\R^{n\times r}$ such that $x(t)\approx V\widehat{x}(t)$ holds. Moreover, we pick
 $W\in\R^{n\times r}$ so that $W^TV$ is nonsingular. Then, \eqref{eq:BQO} leads  to \eqref{eq:BQOr} by inserting $x(t)\approx V\widehat{x}(t)$ and projecting on the left by $(W^TV)^{-1}W^T$. The reduced matrices are given as $\widehat{A}=(W^TV)^{-1}W^TAV$, $\widehat{B}=(W^TV)^{-1}W^TB$, $\widehat{C}=CV$, $\widehat{N}_k=(W^TV)^{-1}W^TN_kV$, and $\widehat{M}_j=V^TM_jV$ for all $k=1,\dots,m$ and $j=1,\dots,p$.

Finding good $V$ and $W$ for a reliable and high-fidelity projection is then the key in model reduction here. One method for obtaining suitable $V$ and $W$ is balanced truncation, which focuses on the concepts of reachability and observability and is studied in \cite{FasGP25,padhi_2024} for BQO systems. Another ansatz is the $\mathcal{H}_2$ optimal model reduction investigated in \cite{FasGP26} for BQO systems. Earlier work on $\sH_2$ optimal reduction is given by, e.g., \cite{ant_beattie_gug_2020,gugercin_2008,xu_optimal_2011} for linear systems, \cite{zhang_2002,benner_birka_2012,flagg_multipoint_2015,redmann_bilinear_2021} for bilinear systems, and also \cite{reiter_2024} for LQO systems. These publications play an important role as a basis for the studies here.

We present an approach that chooses $V$ and $W$ based on an $\mathcal{H}_2$ error analysis for polynomial BQO systems; see \Cref{sec:4}. This framework allows us to compute $V$ and $W$ via a finite sequence of standard Sylvester equations. In contrast to that, the theory for $\mathcal{H}_2$ optimal model reduction for BQO systems (not polynomial) requires solving generalized Sylvester equations; see \Cref{sec:H2nonpolynomial} or \cite{FasGP26}. Our numerical experiments show that the projection based on $\mathcal{H}_2$ theory for polynomial BQO systems yields high-fidelity approximants.

The main contributions of this work are:
\begin{itemize}
    \item We introduce polynomial BQO systems and their $\mathcal{H}_2$ norm, consistent with the $\mathcal{H}_2$ norm for BQO systems (\Cref{def:polyBQO}).
    \item We derive the first-order optimality conditions for $\mathcal{H}_2$ optimal model reduction for polynomial BQO systems (\Cref{thm:gradN}).
    \item We develop an algorithm for constructing reduced BQO systems that approximately satisfy these optimality conditions (\Cref{alg:T_BQO_TSIA}).
\end{itemize}

 The structure of the paper is given as follows: In \Cref{sec:2}, we briefly discuss preliminaries and two propositions for later derivations. 
  \Cref{sec:3} reviews the generalized Sylvester and Lyapunov equations, together with their solutions, that are used throughout the paper, and the $\mathcal{H}_2$ approximation problem of non-polynomial BQO systems in \Cref{sec:H2nonpolynomial}.
In \Cref{sec:4}, polynomial BQO systems are introduced, their $\mathcal{H}_2$ norm is defined, leading to an error expression for the $\mathcal{H}_2$ approximation of polynomial BQO systems. This error expression is then used to derive the first-order optimality conditions for $\mathcal{H}_2$ optimal model reduction of polynomial BQO systems. Then based on these optimality conditions,
\Cref{sec:Alg} develops  \Cref{alg:T_BQO_TSIA} to construct the corresponding reduced model. Finally, we present the numerical experiments in \Cref{sec:5} where we compare the new \Cref{alg:T_BQO_TSIA} with the existing $\mathcal{H}_2$ optimal methods for model reduction of BQO systems.

\section{Preliminaries}\label{sec:2}

Following the standard Householder notation, we label matrices with upper-case Roman letters ($A, B,$ etc.), vectors with lower-case Roman letters ($a, b,$ etc.), and scalars with lower-case Greek letters ($\alpha, \beta$, etc.). 
In the following, we will use the spectral norm $\norm{\cdot}=\norm{\cdot}_2$. 
Also, we denote a symmetric positive semidefinite matrix as $X \geq 0$, while we write $X > 0$ for a symmetric positive definite matrix $X$.
  Moreover, we call $A$ stable if all eigenvalues of $A$ lie in the open left half-plane. Then, there exist $\beta >0$ and $0<\alpha \leq -\max_i(Re(\lambda_i(A)))$ such that $\norm{e^{At}}\leq \beta e^{-\alpha t}$ for $t\geq 0$ \cite{adrianova_1995,willems70}. Next, we write $\tr(A)$ for the trace of a square matrix $A$ and $\vect(X)$ for the vectorization of a rectangular matrix $X$.
Furthermore, we will make use of the standard Kronecker product notation as defined in, e.g., \cite[Section 4.2,4.3]{HorJ94} or \cite[pp.27-28]{GolubVanLoan4th}.

We recall a lemma from \cite{mlinaric_structure-preserving} for simplifying the computation of gradients of error expressions that consist of traces. This will be essential for later use.

\begin{proposition}{\cite[Lemma 2.43]{mlinaric_structure-preserving}}\label[proposition]{lem:gradtr}
    Let $A\in\R^{n\times m}, B\in \R^{n\times n}, C \in \R^{m\times m}$, and $X\in\R^{n\times m}$. Then $\nabla_{X}\tr(AX^T)=A$ and $\nabla_{X}\tr(BXCX^T)=BXC+B^TXC^T$.    
\end{proposition}
Using \Cref{lem:gradtr}, we obtain the following result focusing on the derivatives of block matrices.
\begin{proposition}\label[proposition]{lem:gradtrblock}
    Let $G=G^T= \left[\begin{smallmatrix}
        G_1 & G_2 \\ G_2^T & G_4
    \end{smallmatrix}\right]$, $G_1\in\R^{n\times n}$, $G_2\in\R^{n\times r}$, $G_4\in\R^{r\times r}$ and $H=H^T= \left[ \begin{smallmatrix}
        H_1 & H_2 \\ H_2^T & H_4
    \end{smallmatrix}\right]$, $H_1\in\R^{n\times n}, H_2\in\R^{n\times r}, H_4\in\R^{r\times r}$. Then it holds
    \begin{align}
        \nabla_X \tr\left(\begin{bmatrix}
            Y & \\ & X
        \end{bmatrix} H\right) &= H_4, \label{eq:prop21}\\
        \nabla_X \tr\left(G\begin{bmatrix}
            Y & \\ & X
        \end{bmatrix} H \begin{bmatrix}
            Y^T & \\ & X^T
        \end{bmatrix}\right) &= 2G_2^TYH_2+2G_4XH_4, \label{eq:prop22}\\
        \nabla_W \tr \left( \begin{bmatrix}
            V & W
        \end{bmatrix} G \begin{bmatrix}
            V^T  \\ W^T
        \end{bmatrix} \right) &= 2VG_2+2WG_4,     \label{eq:prop23}   
    \end{align}
    for $Y\in\R^{m\times m}$, $X\in\R^{r\times r}$, $V\in\R^{p\times n}$, and $W\in\R^{p\times r}$.
\end{proposition}

\begin{proof}
For this proof, we make use of \Cref{lem:gradtr}. To show \eqref{eq:prop21}, note that
\begin{align*}
    \nabla_X \tr\left(\begin{bmatrix}
            Y & \\ & X
        \end{bmatrix} H\right) = \nabla_X \tr\left(\begin{bmatrix}
            YH_1 & YH_2 \\ XH_2^T & XH_4
        \end{bmatrix} \right) =\nabla_X \tr\left(
            YH_1 + XH_4
         \right) = H_4.
\end{align*}
Next, \eqref{eq:prop22} holds due to
\begin{align*}
    &\nabla_X \tr\left(G\begin{bmatrix}
            Y & \\ & X
        \end{bmatrix} H \begin{bmatrix}
            Y^T & \\ & X^T
        \end{bmatrix}\right) =\nabla_X \tr\left(\begin{bmatrix}
            G_1Y & G_2X \\ G_2^TY& G_4X
        \end{bmatrix}  \begin{bmatrix}
            H_1Y^T & H_2X^T \\ H_2^T Y^T&  H_4X^T
        \end{bmatrix}\right) \\
        &\qquad =\nabla_X \tr\left(
            G_1YH_1Y^T+G_2XH_2^TY^T  +G_2^TYH_2X^T+G_4XH_4X^T \right) =  2G_2^TYH_2+2G_4XH_4.
\end{align*}
Finally, we can show \eqref{eq:prop23} by observing
\begin{align*}
&\nabla_W \tr \left( \begin{bmatrix}
            V & W
        \end{bmatrix} G \begin{bmatrix}
            V^T  \\ W^T
        \end{bmatrix} \right) = \nabla_W \tr \left( \begin{bmatrix}
            V & W
        \end{bmatrix} \begin{bmatrix}
            G_1V^T+G_2W^T  \\ G_2^TV^T+G_4W^T
        \end{bmatrix} \right) \\
        &  \qquad = \nabla_W \tr \left(         
            VG_1V^T+VG_2W^T + WG_2^TV^T+WG_4W^T
         \right)= 2VG_2+2WG_4.    
\end{align*}
    
\end{proof}
Finally, we state a brief proposition regarding the convergence of a variant of a series.
\begin{proposition} \label{lem:akbk}
    Let $(a_k)_{k\in\mathbb{N}}\subset \R$ be a convergent series, i.e., there exists $\nu_a\in\R$ such that $\lim_{K\to\infty} \sum_{k=1}^Ka_k=\nu_a$. Next, let $(b_k)_{k\in\mathbb{N}}\subset \R$ be a convergent sequence, i.e., there exists $\nu_b\in\R$ such that $\lim_{k\to\infty} b_k =\nu_b$. Then, it holds that $\lim_{K\to\infty}\sum_{k=1}^Ka_{K-k+1}b_k=\nu_a\nu_b$.
\end{proposition}

\section{System theoretical preliminaries}

In this section, we recall the basic facts related to generalized Sylvester and Lyapunov equations appearing in the analysis of BQO systems and the previous work on the $\mathcal{H}_2$ approximation problem of non-polynomial BQO systems.  \label{sec:3}

\subsection{Sylvester and Lyapunov equations} \label{sec:sylvester_bqo}
We start with recalling the most relevant facts regarding the Gramians and the associated Sylvester and Lyapunov equations from \cite{FasGP25,FasGP26}.  Let $\mathcal{S}=(A,B,C,\mathcal{N},\mathcal{M})$ be a BQO system with a stable $A$. As shortly mentioned in the previous section, that is, there exist $\beta>0$ and $0<\alpha \leq -\max_i(\text{Re}(\lambda_i(A)))$ such that $\norm{e^{At}}\leq \beta e^{-\alpha t}$ for $t\geq 0$. In addition, let $\Gamma_1<2\alpha/\beta^2$ with $\Gamma_1:=\max\left\{\norm{\mathcal{N}}^2,\norm{\mathcal{N}_T}^2\right\}$. 
Thus, according to \cite{FasGP25,FasGP26}, the generalized Lyapunov equations 
\begin{subequations}\label{eq:PQ_full}
\begin{align}
    AP+PA^T+\sum_{k=1}^m N_kPN_k^T+BB^T=0,\label{eq:Pfull}\\
    A^TQ+QA+\sum_{k=1}^m N_k^TQN_k+\sum_{j=1}^pM_jPM_j +C^TC=0,\label{eq:Qfull}
\end{align}
\end{subequations}
possess unique symmetric positive semidefinite solutions $P$ and $Q$; see  \cite[Theorem 4.1, Theorem 4.4]{FasGP25}. 
In~\cite{FasGP25}, we define various observability Gramians. Here, the 
observability Gramian $Q$ we use corresponds to 
$Q=Q^S$ in~\cite{FasGP25}.

As discussed in \cite[Remark 3.2]{FasGP26},  $P$ is a Gramian in the true sense, but $Q$ is only a Gramian if additionally $\norm{\mathcal{M}}^2<2\alpha/\beta^2$ holds. See \cite{FasGP25} for further information on this topic. For us, this is irrelevant, since our analysis is based entirely on the Gramian $P$. 

Analogously, let $\widehat{\mathcal{S}}=(\widehat{A},\widehat{B},\widehat{C},\mathcal{\widehat{N}},\mathcal{\widehat{M}})$, given in \eqref{eq:BQOr}, be a reduced BQO system with a stable $\widehat{A}$. Thus, there exist $\widehat{\beta}>0$ and $0<\widehat{\alpha} \leq -\max_i(\text{Re}(\lambda_i(\widehat{A})))$ such that $\norm{e^{\widehat{A}t}}\leq \widehat{\beta} e^{-\widehat{\alpha} t}$ for $t\geq 0$. Additionally, let $\Gamma_2<2\widehat{\alpha}/\widehat{\beta}^2$ with $\Gamma_2:=\max\left\{\norm{\widehat{\mathcal{N}}}^2,\norm{\widehat{\mathcal{N}}_T}^2\right\}$. Then, the generalized Lyapunov equations
\begin{subequations}\label{eq:PQ_red}
\begin{align}
    \widehat{A}\widehat{P}+\widehat{P}\widehat{A}^T+\sum_{k=1}^m \widehat{N}_k\widehat{P}\widehat{N}_k^T+\widehat{B}\widehat{B}^T=0, \label{eq:P_red}\\
    \widehat{A}^T\widehat{Q}+\widehat{Q}\widehat{A}+\sum_{k=1}^m \widehat{N}_k^T\widehat{Q}\widehat{N}_k+\sum_{j=1}^p\widehat{M}_j\widehat{P}\widehat{M}_j +\widehat{C}^T\widehat{C}=0, \label{eq:Q_red}
\end{align}
\end{subequations}
possess unique symmetric positive semidefinite solutions $\widehat{P}$ and $\widehat{Q}$. 
As discussed above for $Q$, $\widehat{Q}$ as a solution to \eqref{eq:Q_red} becomes the true observability Gramian of the reduced system $\widehat{\mathcal{S}}$ if $\norm{\widehat{\mathcal{M}}}^2<2\widehat{\alpha}/\widehat{\beta}^2$.

We will also need Sylvester equations that couple the system matrices of the original system $\mathcal{S}$ and the reduced system $\widehat{\mathcal{S}}$. From \cite[Theorem 3.4]{FasGP26} we know that the equations
\begin{subequations}\label{eq:PQ_mix}
\begin{align}
    AX+X\widehat{A}^T+\sum_{k=1}^m N_kX\widehat{N}_k^T+B\widehat{B}^T=0, \label{eq:P_mix}\\
    A^TY+Y\widehat{A}+\sum_{k=1}^m N_k^TY\widehat{N}_k-\sum_{j=1}^pM_jX\widehat{M}_j -C^T\widehat{C}=0, \label{eq:Q_mix}
\end{align}
\end{subequations}
have unique solutions $X, Y\in\R^{n\times r}$ provided that  $A$ and $\widehat{A}$ are stable, $\norm{\mathcal{N}}^2,\norm{\mathcal{N}_T}^2<2\alpha/\beta^2$, and $\norm{\widehat{\mathcal{N}}}^2,\norm{\widehat{\mathcal{N}}_T}^2<2\widehat{\alpha}/\widehat{\beta}^2$.
Then the state-space representation of the error system $\mathcal{S}_e=\mathcal{S}-\widehat{\mathcal{S}}=(A_e,B_e,C_e,\mathcal{N}_e,\mathcal{M}_e)$  between the 
full and reduced BQO systems can be directly written as $A_e=\diag\left(A,\widehat{A}\right)$, $B_e=\left[ \begin{smallmatrix}
    B \\ \widehat{B}
\end{smallmatrix}\right]$, $C_e=\left[C \; -\widehat{C}\right]$, $\mathcal{N}_e=\left[ N_{e,1} \, \dots \, N_{e,m}\right]$, where $N_{e,k}=\diag\left(N_k,\widehat{N}_k\right)$ for all $k=1,\dots,m$, and $\mathcal{M}_e=\left[ M_{e,1} \, \dots \, M_{e,p}\right]$, where $M_{e,j}=\diag\left(M_j,-\widehat{M}_j\right)$ for all $j\in 1,\dots, p$.
For this error system, we consider the following generalized Lyapunov equations where the solutions can be reconstructed out of the solutions to the previous generalized Lyapunov and Sylvester equations. These equations and their solutions play the fundamental role in defining the $\mathcal{H}_2$ norm for BQO systems. 
\begin{corollary}{\cite[Corollary 3.6]{FasGP26}} \label[corollary]{cor:err_Lyap}
Consider BQO systems $\mathcal{S}$ and $\widehat{\mathcal{S}}$ with stable $A$ and $\widehat{A}$ with stability parameters $\alpha,\beta,\widehat{\alpha},\widehat{\beta}$ such that $\max\left\{\norm{\mathcal{N}}^2,\norm{\mathcal{N}_T}^2\right\} < 2\alpha/\beta^2$ and $\max\left\{\norm{\widehat{\mathcal{N}}}^2,\norm{\widehat{\mathcal{N}}_T}^2\right\} < 2\widehat{\alpha}/\widehat{\beta}^2$.
Then it holds
\begin{align*}
    A_eP_e+P_eA_e^T+&\sum_{k=1}^mN_{e,k}P_eN_{e,k}^T+ B_eB_e^T =0,\\
    A_e^TQ_e+Q_eA_e+&\sum_{k=1}^mN_{e,k}^TQ_eN_{e,k}  + \sum_{j=1}^pM_{e,j}P_eM_{e,j} + C_e^TC_e=0,
\end{align*}
with unique $P_e=\begin{bmatrix}
    P & X\\ X^T &\widehat{P}
\end{bmatrix}$ and $Q_e=\begin{bmatrix}
    Q & Y\\ Y^T &\widehat{Q}
\end{bmatrix}$ and unique solutions $P,Q,\widehat{P},\widehat{Q},X,Y$ to \eqref{eq:Pfull}, \eqref{eq:Qfull}, \eqref{eq:P_red}, \eqref{eq:Q_red}, \eqref{eq:P_mix}, and \eqref{eq:Q_mix}.
\end{corollary}

In the following, we assume that  the relevant  generalized Sylvester and Lyapunov equations have
unique solutions.

\subsection{$\sH_2$ optimal reduction for BQO systems}
\label{sec:H2nonpolynomial}
Now, we briefly recall $\sH_2$ optimal reduction for BQO systems as it is discussed in \cite{FasGP26}. This reference is the foundation for the current manuscript.

In \cite{FasGP26}, we studied 
\emph{non-polynomial} BQO systems, and developed their $\sH_2$ analysis and  $\sH_2$ optimal model reduction framework. Following~\cite{FasGP26}, we introduce the  kernels $h_{C,i}$ and $h_{M,i,j}$ for the BQO system \eqref{eq:BQO}  as 
\begin{subequations}     \label{eq:volt_kern}
 \begin{align}
     &h_{C,i}(t_1,\dots, t_i):=C\bar{P}_i(t_1,\dots, t_i)\in\R^{p\times m^i},  \quad \text{ for } \, i\geq 1,\\
     &h_{M,i,j}(t_1,\dots,t_i,\bar{t}_1,\dots \bar{t}_j) := \mathfrak{M}\left(\bar{P}_i(t_1,\dots, t_i)\otimes \bar{P}_j(\bar{t}_1,\dots, \bar{t}_j)\right) \in\R^{p\times m^{i+j}},  \quad \text{ for } \, i,j\geq 1,
 \end{align}
 \end{subequations}
 where 
 \begin{align*}\mathfrak{M}&=\begin{bmatrix}\vect(M_1) \, \dots \, \vect(M_p)\end{bmatrix}^T\in\R^{p\times n^2}, \\ 
     \bar{P}_1(t_1)&=e^{At_1}B \in \R^{n\times m}, \\
     \bar{P}_i(t_1,\dots ,t_i) &= e^{At_i}\mathcal{N}\left(I_m\otimes \bar{P}_{i-1}(t_1,\dots,t_{i-1})\right) \in\R^{n\times m^i}, \text{ for } i >1.
 \end{align*}
 Then, \cite{FasGP26} shows that the input-output map $u\to y$ for $\mathcal{S}$ can be written as 
\begin{align}
        y(t)&= \sum_{i=1}^\infty \int_0^\infty \dots \int_0^\infty h_{C,i}(t_1,\dots,t_i) (u\otimes \dots \otimes u) dt_1 \dots dt_i \notag \\
        &\qquad +\sum_{i,j=1}^\infty \int_0^\infty \dots \int_0^\infty h_{M,i,j}(t_1,\dots,t_i,\bar{t}_1,\dots,\bar{t}_j) 
          (u\otimes \dots \otimes u) dt_1 \dots dt_i d\bar{t}_1 \dots d\bar{t}_j, \label{eq:iomap}
    \end{align}
using the kernels $h_{C,i}$ and $h_{M,i,j}$ defined in \eqref{eq:volt_kern}. The similar derivation applies to the reduced BQO system \eqref{eq:BQOr}.

Next, based on the formula~\eqref{eq:iomap}, \cite{FasGP26} introduces the $\sH_2$ norm for BQO systems as
\begin{align}
        \norm{\mathcal{S}}_{\mathcal{H}_2}^2 &:= \sum_{i=1}^\infty \int_0^\infty \dots \int_0^\infty \norm{h_{C,i}}_F^2 dt_1 \dots dt_i 
        + \sum_{i,j=1}^\infty \int_0^\infty \dots \int_0^\infty \norm{h_{M,i,j}}_F^2 dt_1\dots dt_id\bar{t}_1 \dots d\bar{t}_j, \label{eq:defH2normfull}        
    \end{align}
    assuming that all integrals exist; see \cite[Definition 4.1]{FasGP26}. This leads to a system-theoretic error description, namely$\norm{\mathcal{S}-\widehat{\mathcal{S}}}_{\sH_2}$, for computing the distance between two BQO systems. Moreover, \cite[Corollary 4.8]{FasGP26} 
    show that, similar to the linear case, the $\mathcal{H}_2$ norm
    provides an upper bound for the output error $y-\widehat{y}$, namely given as 
    \begin{align}
    \sup_{t\geq 0}\norm{y(t)-\widehat{y}(t)}_2^2 \leq \norm{\mathcal{S}-\widehat{\mathcal{S}}}_{\mathcal{H}_2}^2 \alpha(u), \label{eq:yerr_est}
\end{align}
where $y$ and $\widehat{y}$ are the corresponding outputs of $\mathcal{S}$ and $\widehat{\mathcal{S}}$, respectively, when the same input $u$ is applied to both systems. Here, the term $\alpha(u)$ describes the dependency on the input $u$. This motivates developing $\mathcal{H}_2$ optimal model order reduction approaches because a small $\mathcal{H}_2$ error provides us with a small estimate on the output error between the original and a reduced BQO system. 

Via gradient derivations, \cite{FasGP26} derives the first-order optimality conditions for  
$\mathcal{H}_2$ optimal model reduction of BQO systems as stated in the next result. 
\begin{theorem}{\cite[Theorem 5.1]{FasGP26}} \label{thm:grad}
    Given the full BQO system $\mathcal{S}$ in~\eqref{eq:BQO}, assume that the reduced BQO system $\widehat{\mathcal{S}}$ in~ \eqref{eq:BQOr} is locally $\mathcal{H}_2$ optimal in the sense that $\widehat{\mathcal{S}}$ is a local minimum of $\norm{\mathcal{S}-\widehat{\mathcal{S}}}_{\mathcal{H}_2}$. Furthermore, suppose that $\widehat{P},X$ are the unique solutions to \eqref{eq:P_red} and \eqref{eq:P_mix} and also suppose that $\Psi=\Psi^T\in\R^{r\times r}$ and $\Pi\in\R^{n\times r}$ are the unique solutions to the generalized Lyapunov and Sylvester equations
     \begin{subequations}
    \begin{align}
        \widehat{A}^T\Psi+\Psi\widehat{A}+\sum_{k=1}^m\widehat{N}_k^T\Psi\widehat{N}_k+2\sum_{j=1}^p\widehat{M}_j\widehat{P}\widehat{M}_j+\widehat{C}^T\widehat{C}&=0, \label{eq:lyap_psi} \\
        A^T\Pi+\Pi\widehat{A}+\sum_{k=1}^mN_k^T\Pi\widehat{N}_k+2\sum_{j=1}^pM_jX\widehat{M}_j+C^T\widehat{C}&=0. \label{eq:lyap_phi}
    \end{align}
     \end{subequations}
    Then the following optimality conditions of first order hold \\
    \begin{subequations}\label{eq:optcond}
    \noindent\begin{minipage}{.5\linewidth}    
    \begin{align}
        0 &= \Psi\widehat{P}-\Pi^TX, \label{eq:optcond1}\\
        0 &= \Psi\widehat{B}-\Pi^TB,\label{eq:optcond2}\\
        0 &= \Psi\widehat{N}_k\widehat{P}-\Pi^TN_kX,\label{eq:optcond3}        
    \end{align}    
    \end{minipage}
    \begin{minipage}{.45\linewidth}
    \begin{align}
    0 &= \widehat{C}\widehat{P}-CX,\label{eq:optcond4}\\
        0 &= \widehat{P}\widehat{M}_j\widehat{P}-X^TM_jX,\label{eq:optcond5}
    \end{align}
    \end{minipage}
    \end{subequations}  \\ 

    \noindent
    for all $k=1,\dots, m$ and $j=1,\dots, p$.
    
\end{theorem}
To enforce these optimality conditions, \cite{FasGP26} developed the iterative algorithm \texttt{BQO\_TSIA} (\cite[Algorithm 5.1]{FasGP26}) which requires solving the generalized Sylvester equations
\begin{subequations}
\label{eq:BQO_TSIA}
    \begin{align}
    AX+X\widehat{A}^T+\sum_{k=1}^m N_k X \widehat{N}_k^T + B\widehat{B}^T &=0, \\
    A^T\Pi+\Pi\widehat{A}+\sum_{k=1}^mN_k^T \Pi \widehat{N}_k + 2 \sum_{j=1}^p M_j X \widehat{M}_j + C^T\widehat{C} &=0,
    \end{align}
\end{subequations}
 and uses these solutions $X$ and $\Pi$ for a Petrov-Galerkin projection of the original system in each step to obtain a new reduced BQO system. The computational bottleneck of \texttt{BQO\_TSIA} is the need to repeatedly solve \eqref{eq:BQO_TSIA} with different reduced system matrices at every step. This becomes computationally challenging for large dimensions. 
 In the next section, we follow a new approach that introduces polynomial BQO systems and develop the $\mathcal{H}_2$ theory in that setting. We first derive the $\mathcal{H}_2$ optimality conditions for these systems. The corresponding iterative algorithm, \Cref{alg:T_BQO_TSIA} no longer requires solving generalized Sylvester equations. It only relies on a few solutions of standard Sylvester equations which greatly improves the computational cost.

\section{$\sH_2$ optimal reduction for polynomial BQO systems} \label{sec:4}

The $\mathcal{H}_2$ analysis and $\mathcal{H}_2$ optimal model reduction framework of~\cite{FasGP26} worked with the full Volterra series formulation of the original BQO system.  In this section, we 
define the polynomial BQO system and extend this analysis to this setting. 
We start by approximating the original BQO system by a polynomial BQO system and use this approximation to construct a reduced-order model of the original system. More precisely, the polynomial BQO system is obtained by truncating the associated Volterra series. The resulting finite Volterra series representation allows us to work with finite expressions, which simplifies $\mathcal{H}_2$ norms and leads to a finite sequence of standard Lyapunov and Sylvester equations to solve instead of generalized Lyapunov and Sylvester equations; see \cite{FasGP26}. Based on these expressions, we determine an $\mathcal{H}_2$ optimal reduced-order model for the polynomial BQO system. 

\subsection{Polynomial BQO systems}
Analogous to the bilinear case \cite[Section 5.1]{flagg_multipoint_2015}, we can define polynomial BQO systems. These systems are based on the same kernels \eqref{eq:volt_kern} as original BQO systems and are obtained by truncating the two infinite sums in the input-output map \eqref{eq:iomap}. For simplicity, we use the same truncation order $K$ for both sums, although different truncation orders could be chosen. 
Polynomial systems are often a good choice for an approximation of systems that are based on an infinite sequence of kernels where higher order terms include multiple exponential terms of the kind $e^{At}$. These terms of higher order have to be decreasing in a suitable norm, especially for later times, to ensure convergence of the associated series. Hence, it is often sufficient to consider only a few first terms and cut off after a certain truncation order.

\begin{definition} \label{def:polyBQO}
    Let $\mathcal{S}=(A,B,C,\mathcal{N},\mathcal{M})$ be a BQO system and $K\in\mathbb{N}$. Then we define the polynomial BQO system $\mathcal{S}^K$ as the input-output map $\mathcal{S}^K:u\to y_K$ given by
    \begin{align*}
        y_K(t)&= \sum_{i=1}^K \int_0^\infty \dots \int_0^\infty h_{C,i}(t_1,\dots,t_i) (u\otimes \dots \otimes u) dt_1 \dots dt_i \\
        &\qquad +\sum_{i,j=1}^K \int_0^\infty \dots \int_0^\infty h_{M,i,j}(t_1,\dots,t_i,\bar{t}_1,\dots,\bar{t}_j)          (u\otimes \dots \otimes u) dt_1 \dots dt_i d\bar{t}_1 \dots d\bar{t}_j, 
    \end{align*}
    with the kernels $h_{C,i}$ and $h_{M,i,j}$ defined in \eqref{eq:volt_kern}.
    We define the corresponding $\mathcal{H}_2$ norm of $\mathcal{S}^K$ as
    \begin{align*}
        \norm{\mathcal{S}^K}_{\mathcal{H}_2}^2 &:= \sum_{i=1}^K \int_0^\infty \dots \int_0^\infty \norm{h_{C,i}}_F^2 dt_1 \dots dt_i 
        + \sum_{i,j=1}^K \int_0^\infty \dots \int_0^\infty \norm{h_{M,i,j}}_F^2 dt_1\dots dt_id\bar{t}_1 \dots d\bar{t}_j,
    \end{align*}
 assuming that all integrals exist.
\end{definition}

The definition of the $\sH_2$ norm for polynomial BQO systems is consistent with the definition of the $\sH_2$ norm for BQO systems; see \eqref{eq:defH2normfull}. In fact, the two definitions are identical if the infinite sums in the latter are truncated by setting the higher-order kernels $h_{C,i}$ and $h_{M,i,j}$ to zero for $i,j>K$. 

It follows directly that $\mathcal{S}^K\to\mathcal{S}$ and $\norm{\mathcal{S}^K}_{\mathcal{H}_2} \to \norm{\mathcal{S}}_{\mathcal{H}_2}$ as $K\to \infty.$ Thus, the polynomial BQO system $\mathcal{S}^K$ provides an approximation to the original BQO system $\mathcal{S}$ that converges to $\mathcal{S}$ in the $\mathcal{H}_2$ norm. Analogously, we can define a polynomial BQO system $\widehat{\mathcal{S}}^K$ of reduced dimension and study its corresponding $\mathcal{H}_2$ norm. 

To derive $\mathcal{H}_2$ optimal reduction of a polynomial BQO system, we first need a suitable expression for the error
 $\norm{\mathcal{S}^K-\widehat{\mathcal{S}}^K}_{\mathcal{H}_2}$. Our derivation is based on \cite{FasGP26}.

\begin{lemma}    
Consider the two polynomial BQO systems $\mathcal{S}^K$ and $\widehat{\mathcal{S}}^K$. 
Define $R^{(K)}:=\left[\begin{smallmatrix}
    P^{(K)} & X^{(K)} \\ \left(X^{(K)}\right)^T & \widehat{P}^{(K)}
\end{smallmatrix}\right]$ 
where $P^{(K)}:=\sum_{i=1}^KP_i$, $X^{(K)}:=\sum_{i=1}^K X_i$, and $\widehat{P}^{(K)}:=\sum_{i=1}^K\widehat{P}_i$ with $P_i$,  $X_i$, $\widehat{P}_i$, and $\widehat{X}_i$ being unique solutions to
\begin{subequations}
\label{eq:lyap_collection}
\begin{align}
    0&= AP_1+P_1A^T+BB^T, \\
    0&= AP_i+P_iA^T+\sum_{k=1}^mN_k P_{i-1}N_k^T, \quad \forall \, i=2,\dots, K, \\
    0&= AX_1+X_1\widehat{A}^T+B\widehat{B}^T, \\
    0&= AX_i+X_i\widehat{A}^T+\sum_{k=1}^mN_k X_{i-1}\widehat{N}_k^T, \quad \forall \, i=2, \dots, K, \\
    0&= \widehat{A}\widehat{P}_1+\widehat{P}_1\widehat{A}^T+\widehat{B}\widehat{B}^T, \\
    0&= \widehat{A}\widehat{P}_i+\widehat{P}_i\widehat{A}^T+\sum_{k=1}^m\widehat{N}_k \widehat{P}_{i-1}\widehat{N}_k^T, \quad \forall \, i=2,\dots,K.    
\end{align}    
\end{subequations}
Then, the $\sH_2$ distance between $\mathcal{S}^K$ and $\widehat{\mathcal{S}}^K$ is given by 
\begin{align} 
    \qquad &\norm{\mathcal{S}^K-\widehat{\mathcal{S}}^K}_{\mathcal{H}_2}^2 =  \norm{\mathcal{S}_e^K}_{\mathcal{H}_2}^2 
    =
    \tr\left( C_e R^{(K)} C_e^T \right) + \sum_{j=1}^p \tr \left( R^{(K)} M_{e,j} R^{(K)} M_{e,j} \right) \label{eq:SerrN}.
\end{align}
\end{lemma}
\begin{proof}
First, we study the polynomial BQO system $\mathcal{S}^K$ and its $\mathcal{H}_2$ norm. It holds from \Cref{def:polyBQO} that
\begin{align*}\norm{\mathcal{S}^K}_{\mathcal{H}_2}^2&=\sum_{i=1}^K \int_0^\infty \dots \int_0^\infty \norm{h_{C,i}}_F^2 dt_1 \dots dt_i 
        + \sum_{i,j=1}^K \int_0^\infty \dots \int_0^\infty \norm{h_{M,i,j}}_F^2 dt_1\dots dt_id\bar{t}_1 \dots d\bar{t}_j \\
        &= \sum_{i=1}^K \int_0^\infty \dots \int_0^\infty \tr\left(h_{C,i}h_{C,i}^T\right)  dt_1 \dots dt_i 
        + \sum_{i,j=1}^K \int_0^\infty \dots \int_0^\infty \tr \left(h_{M,i,j}h_{M,i,j}^T\right) dt_1\dots dt_id\bar{t}_1 \dots d\bar{t}_j \\
        &= \sum_{i=1}^K \int_0^\infty \dots \int_0^\infty \tr\left(C\bar{P}_i\bar{P}_i^T C^T\right)  dt_1 \dots dt_i \\
        & \qquad + \sum_{i,j=1}^K \int_0^\infty \dots \int_0^\infty \tr \left(\mathfrak{M}(\bar{P}_i\otimes \bar{P}_j)(\bar{P}_i\otimes\bar{P}_j)^T\mathfrak{M}^T\right) dt_1\dots dt_id\bar{t}_1 \dots d\bar{t}_j \\
        &=\sum_{i=1}^K \tr (CP_iC^T) + \sum_{i,j=1}^K \tr\left(\mathfrak{M}(P_i\otimes P_j)\mathfrak{M}^T\right) 
        = \tr \left(CP^{(K)}C^T\right) + \tr\left(\mathfrak{M}(P^{(K)}\otimes P^{(K)})\mathfrak{M}^T\right),
\end{align*}
where we used the definition of the kernels $h_{C,i}$ and $h_{M,i,j}$ in \eqref{eq:volt_kern}. Finally, \cite[Corollary 2.2]{FasGP26} yields $\tr\left(\mathfrak{M}(P^{(K)}\otimes P^{(K)})\mathfrak{M}^T\right)=\sum_{j=1}^p \tr\left(P^{(K)}M_jP^{(K)}M_j\right)$.
Then, applying the same analysis to 
$\widehat{\mathcal{S}}^K$ and to the error system $\mathcal{S}^K_e=\mathcal{S}^K-\widehat{\mathcal{S}}^K$ leads to the desired statement.    
\end{proof}

\begin{remark}
For $R^{(K)}=\left[\begin{smallmatrix}
    P^{(K)} & X^{(K)} \\ \left(X^{(K)}\right)^T & \widehat{P}^{(K)}
\end{smallmatrix}\right]$, the following equations hold
\begin{subequations}\label{eq:lyapRN}
\begin{align} 
    0&= A_e R^{(1)} + R^{(1)} A_e^T + B_e B_e^T, \\
    0&= A_e R^{(i)} + R^{(i)} A_e^T + \sum_{k=1}^m N_{e,k} R^{(i-1)} N_{e,k}^T + B_eB_e^T,
\end{align}
\end{subequations}
for all $i=2,\dots,K$, as a direct consequence of \eqref{eq:lyap_collection} where $A_e, B_e$ and $N_{e,k}$ are matrices of the error system $\mathcal{S}_e$ defined above \Cref{cor:err_Lyap}.
 Based on \Cref{cor:err_Lyap}, it is clear that $R^{(K)}\to P_e =\left[\begin{smallmatrix}
    P & X \\ X^T & \widehat{P}
\end{smallmatrix}\right]$ as $K\to\infty$.
\end{remark}

\subsection{$\sH_2$ optimality for polynomial BQO systems}
In this section, we search for an $\sH_2$ optimal polynomial BQO system that is a local minimum of the error expression $\norm{\mathcal{S}^K-\widehat{\mathcal{S}}^K}_{\mathcal{H}_2}^2$ in \eqref{eq:SerrN} 
and derive the optimality conditions for an $\sH_2$ optimal reduced polynomial BQO system. As we will show, our analysis naturally complements the full Volterra series 
approach of \Cref{thm:grad} and \cite[Section 5]{FasGP26}. 
\begin{theorem}\label{thm:gradN}
 Consider the polynomial BQO system $\mathcal{S}^K$ and assume that the reduced polynomial BQO system $\widehat{\mathcal{S}}^K$ is locally $\sH_2$ optimal in the sense of minimizing \eqref{eq:SerrN}. Furthermore, let $\widehat{P}^{(i)}=\sum_{j=1}^i\widehat{P}_j,X^{(i)}=\sum_{j=1}^iX_j$, for all $i=1,\dots,K$, with the unique solutions $\widehat{P}_j$, $X_j$ to \eqref{eq:lyap_collection}. Also, let $\Psi_i=\Psi_i^T\in \R^{r\times r}$ and $\Pi_i\in \R^{n\times r}$ be the unique solutions to the generalized Lyapunov and Sylvester equations
    \begin{subequations}\label{eq:lyapPhiPsi}
    \begin{align}
        0&=\widehat{A}^T\Psi_1+\Psi_1\widehat{A}+2\sum_{j=1}^p\widehat{M}_j\widehat{P}^{(K)}\widehat{M}_j+\widehat{C}^T\widehat{C}, \label{eq:lyap_psi1} \\
        0&=\widehat{A}^T\Psi_{i}+\Psi_{i}\widehat{A}+\sum_{k=1}^m\widehat{N}_k^T \Psi_{i-1} \widehat{N}_k,  \label{eq:lyap_psii} \\
        0&=A^T\Pi_1+\Pi_1\widehat{A}+2\sum_{j=1}^p M_jX^{(K)}\widehat{M}_j+C^T\widehat{C}, \label{eq:lyap_phi1} \\
        0&=A^T\Pi_i+\Pi_i\widehat{A}+\sum_{k=1}^m N_k^T \Pi_{i-1} \widehat{N}_k, \label{eq:lyap_phii} 
    \end{align}
     \end{subequations}
     for all $i=2,\dots,K$.
    Then the following  first order optimality conditions hold:  \begin{subequations}\label{eq:optcondN}
    \begin{align}
        0 &= \sum_{i=1}^K \Psi_{K-i+1}\widehat{P}^{(i)}-\Pi_{K-i+1}^T X^{(i)}, \label{eq:optcondN1}\\
        0 &= \sum_{i=1}^K \Psi_{K-i+1}\widehat{B}- \Pi_{K-i+1}^T B,\label{eq:optcondN2}\\
        0 &= \sum_{i=1}^{K-1} \Psi_{K-i}\widehat{N}_k\widehat{P}^{(i)}-\Pi_{K-i}^TN_kX^{(i)},\label{eq:optcondN3}\\
        0 &= \widehat{C}\widehat{P}^{(K)}-CX^{(K)},\label{eq:optcondN4}\\
        0 &= \widehat{P}^{(K)}\widehat{M}_j\widehat{P}^{(K)}-(X^{(K)})^TM_jX^{(K)}, \label{eq:optcondN5}
    \end{align}
    \end{subequations}
    for all $k=1,\dots, m$ and $j=1,\dots, p$.
\end{theorem}
\begin{proof}
    Similarly to \cite{FasGP26}, we consider the $\sH_2$ optimal model reduction problem as the constrained optimization problem     \begin{align*}\min_{\widehat{A},\widehat{B},\widehat{C},\widehat{N}_k,\widehat{M}_j} & \tr\left( C_e R^{(K)} C_e^T \right) + \sum_{j=1}^p \tr \left( R^{(K)} M_{e,j} R^{(K)} M_{e,j} \right), 
    \quad \text{subject to } \eqref{eq:lyapRN},
    \end{align*}
    where we use the error expression for $\norm{\mathcal{S}^K-\widehat{\mathcal{S}}^K}_{\mathcal{H}_2}^2 $ from \eqref{eq:SerrN}.
    This leads to the associated Lagrange function as
    \begin{align*}
        &\mathfrak{L}(\widehat{A},\widehat{B},\widehat{C},\widehat{N}_k,\widehat{M}_j,R^{(1)},\dots, R^{(K)}, \Lambda_1, \dots, \Lambda_K ) \\
        &= \tr\left( C_e R^{(K)}C_e \right) + \sum_{j=1}^p \tr \left( R^{(K)} M_{e,j} R^{(K)} M_{e,j} \right)  + \tr\left(\Lambda_1^T \left( A_e R^{(1)} +R^{(1)}A_e^T + B_e B_e^T \right) \right) \\
         & \quad + \tr\left(\Lambda_2^T \left(A_e R^{(2)} +R^{(2)}A_e^T +\sum_{k=1}^m N_{e,k} R^{(1)}N_{e,k}^T + B_e B_e^T \right) \right)+ \dots \\
    &  \quad + \tr\left(\Lambda_K^T \left(A_e R^{(K)} +R^{(K)}A_e^T +\sum_{k=1}^m N_{e,k} R^{(K-1)}N_{e,k}^T + B_e B_e^T \right) \right), 
    \end{align*}
where $\Lambda_1,\dots, \Lambda_K \in \mathbb{R}^{(n+r)\times(n+r)}$ are the Lagrange multipliers. Now, we use \Cref{lem:gradtr,lem:gradtrblock} to derive the first order conditions. 
By differentiating the Lagrange function with respect to relevant quantities, we obtain
\begin{align*}
    \nabla_{\widehat{C}}\mathfrak{L} &= 2\widehat{C}\widehat{P}^{(K)}-2CX^{(K)}, \\
    \nabla_{\widehat{M}_j}\mathfrak{L} &= 2 \widehat{P}^{(K)} \widehat{M}_j\widehat{P}^{(K)}- 2(X^{(K)})^T M_j X^{(K)}, \\
    \nabla_{R^{(1)}} \mathfrak{L}&= A_e^T \Lambda_1^T+\Lambda_1^TA_e+ \sum_{k=1}^m N_{e,k}^T \Lambda_2^T N_{e,k}, \\
    & \vdots \\
    \nabla_{R^{(K-1)}} \mathfrak{L}&= A_e^T \Lambda_{K-1}^T+\Lambda_{K-1}^TA_e 
     + \sum_{k=1}^m N_{e,k}^T \Lambda_K^T N_{e,k}, \\
    \nabla_{R^{(K)}} \mathfrak{L}&= A_e^T \Lambda_K^T+\Lambda_K^TA_e 
     + 2\sum_{j=1}^p M_{e,j} R^{(K)}M_{e,j}+ C_e^TC_e.
\end{align*}
Hence, setting the first two equations to zero leads to \eqref{eq:optcondN4} and \eqref{eq:optcondN5}. Moreover, setting the expressions for $\nabla_{R^{(i)}}\mathfrak{L}$ to zero follows $\Lambda_{K-i+1}=\left[ \begin{smallmatrix}
    \ast & -\Pi_i \\
    -\Pi_i^T & \Psi_i
\end{smallmatrix} \right]$ with $\Pi_i$ and $\Psi_i$ as unique solutions to \eqref{eq:lyapPhiPsi}. 
Lastly, computing the gradients 
\begin{align*}
    \nabla_{\widehat{A}}\mathfrak{L} &= 2 \left(\sum_{i=1}^K  \Psi_{K-i+1}\widehat{P}^{(i)} - \Pi_{K-i+1}^TX^{(i)}\right), \\
    \nabla_{\widehat{B}} \mathfrak{L} &= 2\left(\sum_{i=1}^K \Psi_{K-i+1}\widehat{B}-\Pi_{K-i+1}^TB\right), \\
    \nabla_{\widehat{N}_k} \mathfrak{L} &= 2\left(\sum_{i=1}^{K-1} \Psi_{K-i}\widehat{N}_k\widehat{P}^{(i)}- \Pi_{K-i}^TN_k X^{(i)}\right),
\end{align*}
and setting these expressions to zero leads to conditions \eqref{eq:optcondN1}, \eqref{eq:optcondN2}, and \eqref{eq:optcondN3}.
\end{proof}
For comparison with the original $\mathcal{H}_2$ optimality conditions, we state a remark using \Cref{lem:akbk}.

\begin{remark} \label{rem:optlim}
    Analyzing the conditions \eqref{eq:optcondN} in \Cref{thm:gradN}, we observe that \eqref{eq:optcondN} yields
    the optimality conditions of \Cref{thm:grad} for the full untruncated BQO case  as we let $K\to\infty$.  This holds true because $\widehat{P}^{(K)}\to\widehat{P}$, $X^{(K)}\to X$, $\sum_{i=1}^K \Psi_i \to \Psi$, and $\sum_{i=1}^{K}\Pi_i\to \Pi$, for $K\to\infty$. 
    More precisely, \eqref{eq:optcond4} and \eqref{eq:optcond5} follow from \eqref{eq:optcondN4} and \eqref{eq:optcondN5}, respectively. For the remaining equations, we use \Cref{lem:akbk} to show the identities. Hence, the optimality conditions in \Cref{thm:gradN} are consistent with the original $\mathcal{H}_2$ optimality conditions in \Cref{thm:grad}.
    \end{remark}
    The most important distinction from the untruncated case  is that 
    the new optimality conditions \eqref{eq:optcondN} are based on solutions to standard Sylvester equations, in contrast to the original optimality conditions \eqref{eq:optcondN} that require solutions to generalized Sylvester equations. This  will significantly simplify the computational cost of the corresponding algorithm as we discuss next.

\section{Computing $\mathcal{H}_2$ optimal reduced models: Algorithm and numerical examples} \label{sec:Alg}
Having established the theoretical analysis and optimality conditions in~\Cref{sec:4}, we now turn our focus to numerical developments. We first develop an algorithm in~\Cref{sec:alg} based on the optimality conditions and then in \Cref{sec:5} apply the algorithm to two numerical examples. 
\subsection{An $\mathcal{H}_2$ optimal algorithm for polynomial BQO systems}
\label{sec:alg}

Our goal is to find a reduced BQO system \eqref{eq:BQOr} that provides a high-fidelity approximation to the original BQO system. Our chosen measure for a good approximation is the difference in the $\mathcal{H}_2$ norm to perform optimal model order reduction. We shortly recall \eqref{eq:yerr_est}, namely $
    \sup_{t\geq 0}\norm{y(t)-\widehat{y}(t)}_2^2 \leq \norm{\mathcal{S}-\widehat{\mathcal{S}}}_{\mathcal{H}_2}^2 \alpha(u)$.
In this context, inspired by this bound on the output error, we search for a reduced BQO system $\widehat{\mathcal{S}}$ such that $\norm{\mathcal{S}-\widehat{\mathcal{S}}}_{\mathcal{H}_2}$ is small. Note that 
\begin{align*}
    \norm{\mathcal{S}-\widehat{\mathcal{S}}}_{\mathcal{H}_2}
    =\norm{\mathcal{S}-\mathcal{S}^K + \mathcal{S}^K -\widehat{\mathcal{S}}^K+\widehat{\mathcal{S}}^K-\widehat{\mathcal{S}}}_{\mathcal{H}_2} 
    \leq \norm{\mathcal{S}-\mathcal{S}^K}_{\mathcal{H}_2}+\norm{\mathcal{S}^K-\widehat{\mathcal{S}}^K}_{\mathcal{H}_2}+ \norm{\widehat{\mathcal{S}}-\widehat{\mathcal{S}}^K}_{\mathcal{H}_2}.
\end{align*}
We focus on minimizing the dominant term, namely the second term $\norm{\mathcal{S}^K-\widehat{\mathcal{S}}^K}_{\mathcal{H}_2}$, especially when $K$ is large or there exists a fast convergence $\mathcal{S}^K\to\mathcal{S}$ and $\widehat{\mathcal{S}}^K\to\widehat{\mathcal{S}}$ . 
Here, the results from \Cref{sec:4} play an important role. We can study the BQO system \eqref{eq:BQO} as a polynomial BQO system and minimize the term $\norm{\mathcal{S}^K-\widehat{\mathcal{S}}^K}_{\mathcal{H}_2}$. The minimization process is based on the optimality conditions in \Cref{thm:gradN} that are consistent with the original $\mathcal{H}_2$ optimality conditions in \Cref{thm:grad}.

A direct construction of a polynomial BQO system that fulfills the optimality conditions in \Cref{thm:gradN} is not straightforward due to the nested summation structure in the conditions. Hence, we have to find another approach so that we can approximately fulfill these conditions. If we compare the polynomial conditions \eqref{eq:optcondN} with the original conditions \eqref{eq:optcond}, we observe that, e.g., $\widehat{P}^{(K)}$ and $X^{(K)}$ are approximations to $\widehat{P}$ and $X$ from the original case and fulfill similar conditions. In the original setting for full BQO systems, a state space transformation $\widehat{x}=\widehat{P}z$ allows us to write the reduced system as a system given by a Petrov-Galerkin projection; see \cite{FasGP26}. This Petrov-Galerkin projection is based on the matrices $X$ and $\Pi$ from \Cref{thm:grad}. 

So now, similar to \cite{flagg_multipoint_2015,FasGP26}, we use a state space transformation of the reduced system, $\widehat{x}=\widehat{P}^{(K)}z$, and
$X^{(K)}$ as an approximation to $X$ and $\sum_{i=1}^{K}\Pi_i=:\Pi^{(K)}$ as an approximation to $\Pi$ from \Cref{thm:grad}. This results in~\Cref{alg:T_BQO_TSIA}, which we call \emph{Truncated Bilinear quadratic output two-sided iteration algorithm}, abbreviated as \texttt{T-BQO\_TSIA}. The fundamental difference from the algorithm \texttt{BQO\_TSIA}~\cite[Algorithm 5.1]{FasGP26}, which enforces the optimality conditions  of~\Cref{thm:grad} for the original non-polynomial BQO systems is that \texttt{T-BQO\_TSIA} requires solving only regular Sylvester equations as opposed to generalized Sylvester equations, that are computational more demanding. 
The proposed algorithm can also be considered as an extension of the truncated B-IRKA (TB-IRKA) algorithms for bilinear systems (where $M_j=0$ holds); see \cite[Algorithm 2]{flagg_multipoint_2015}. Due to the structure of the conditions \eqref{eq:optcondN},  \Cref{alg:T_BQO_TSIA} enforces only the conditions \eqref{eq:optcondN4} and \eqref{eq:optcondN5} explicitly in each iteration while approximating the conditions \eqref{eq:optcondN1}, \eqref{eq:optcondN2}, and \eqref{eq:optcondN3}. This can be seen by observing that $\widehat{P}^{(K)}$ and $X^{(K)}$ play the roles of $\widehat{P}$ and $X$, respectively. A state space transformation of the reduced system, $\widehat{x}=\widehat{P}^{(K)}z$, yields $\widehat{C}=CX^{(K)}$ and $\widehat{M}_j=(X^{(K)})^TM_jX^{(K)}$ directly from \eqref{eq:optcondN4} and \eqref{eq:optcondN5}. This is exactly the construction of $\widehat{C}$ and $\widehat{M}_j$ in \Cref{alg:T_BQO_TSIA}.

\begin{algorithm}
    \caption{\texttt{Truncated Bilinear quadratic output two-sided iteration algorithm (T-BQO-TSIA)}}\label{alg:T_BQO_TSIA}
    \begin{algorithmic}[1]
    \Require $A, B, C, N_k, M_j$ as in \eqref{eq:BQO}, $\widehat{A}, \widehat{B}, \widehat{C}, \widehat{N}_k, \widehat{M}_j$ as in \eqref{eq:BQOr}, truncation index $K$
    \Ensure $\widehat{A}^{opt}, \widehat{B}^{opt}, \widehat{C}^{opt}, \widehat{N}^{opt}_k, \widehat{M}^{opt}_j$ as in \eqref{eq:BQOr}
    \While{not converged}
        \State For $X^{(1)}, X^{(2)}, \dots, X^{(K)}\in\mathbb{R}^{n\times r}$, solve 
        \begin{align*}
            AX^{(1)}+X^{(1)}\widehat{A}^T+B\widehat{B}^T&=0, \\
            AX^{(i)}+X^{(i)}\widehat{A}^T+\sum_{k=1}^mN_kX^{(i-1)}\widehat{N}_k^T+B\widehat{B}^T&=0, \qquad \text{ for }i=2,\dots, K.        
        \end{align*}
                
        \State 
        For $\Pi^{(1)}, \Pi^{(2)}, \dots, \Pi^{(K)}\in\mathbb{R}^{n\times r}$, solve 
        \begin{align*}
        A^T\Pi^{(1)}+\Pi^{(1)}\widehat{A} +2\sum_{j=1}^pM_jX^{(K)}\widehat{M}_j+ C^T\widehat{C} &=0, \\
        A^T\Pi^{(i)}+\Pi^{(i)}\widehat{A}+\sum_{k=1}^mN_k^T\Pi^{(i-1)}\widehat{N}_k +2\sum_{j=1}^pM_jX^{(K)}\widehat{M}_j + C^T\widehat{C}&=0, \qquad  \text{ for }i=2,\dots, K.
        \end{align*}
        
        \State Compute $V=\orth(X^{(K)})$, $W=\orth(\Pi^{(K)})$.
        \State Compute $\widehat{A}=(W^TV)^{-1}W^TAV$, $\widehat{B}=(W^TV)^{-1}W^TB$, $\widehat{C}=CV$, $\widehat{N}_k=(W^TV)^{-1}W^TN_kV$, $\widehat{M}_j=V^TM_jV$.
    \EndWhile
    \State $\widehat{A}^{opt}=\widehat{A}$, $\widehat{B}^{opt}=\widehat{B}$, $\widehat{C}^{opt}=\widehat{C}$, $\widehat{N}^{opt}_k=\widehat{N}_k$ and $\widehat{M}_j^{opt}=\widehat{M}_j$. 
    \end{algorithmic}
    \end{algorithm} 

    Upon convergence of \Cref{alg:T_BQO_TSIA}, we can see that the non-polynomial optimality conditions \eqref{eq:optcond} are fulfilled as $K\to\infty$. The following corollary is a direct consequence of \cite[Theorem 5.3]{FasGP26} and \Cref{rem:optlim}.
    \begin{corollary}\label{cor:alg1cond}
        Assume \Cref{alg:T_BQO_TSIA} converges. Then $\widehat{A}^{opt}$, $\widehat{B}^{opt}$, $\widehat{C}^{opt}$, $\widehat{N}_k^{opt}$, and $\widehat{M}_j^{opt}$ satisfy the optimality conditions \eqref{eq:optcondN} and \eqref{eq:optcond} as $K\to \infty$.
    \end{corollary}

    In light of \Cref{cor:alg1cond}, \Cref{alg:T_BQO_TSIA} can be interpreted, from an algorithmic perspective, as a variant of \texttt{BQO\_TSIA} where, in each iteration, the Sylvester equations are solved via a fixed-point iteration that is terminated after a fixed number of $K$ steps rather than run to convergence.

\subsection{Numerical experiments} \label{sec:5}
In this section, we apply \Cref{alg:T_BQO_TSIA} on two numerical examples and compare it to the existing full $\sH_2$ optimal algorithm from \cite{FasGP26}. 

For more clarity, we denote the compared algorithms here: 
\begin{itemize}
    \item \texttt{BQO\_TSIA}: \cite[Algorithm 5.1]{FasGP26} where the generalized Sylvester equations are solved via a fixed point iteration of regular Sylvester equations.         
    \item \texttt{BQO\_TSIA\_GLGMRES}: \cite[Algorithm 5.1]{FasGP26} where the generalized Sylvester equations are solved using a global variant of GMRES \cite{jbilou_global_1999} 
    \item \texttt{T-BQO\_TSIA(K)}: \Cref{alg:T_BQO_TSIA} where we only solve a fixed number of regular Sylvester equations per iteration as described in \Cref{alg:T_BQO_TSIA}. We choose here $K=2$ and $K=4$ for our experiments. Higher orders were not needed for our cases to guarantee that the truncated error were close enough to the non-polynomial case. 
\end{itemize} 

Similarly to \cite{FasGP26}, we check the convergence of the iterative algorithms by monitoring the change in the squared relative $\sH_2$ error given as $\abs{\eta^{(j+1)}-\eta^{(j)}}/\eta^{(1)}$ with $\eta^{(j)}=\frac{\norm{\mathcal{S}-\widehat{\mathcal{S}}^{(j)}}_{\sH_2}^2}{\norm{\mathcal{S}}_{\sH_2}^2}$ where we denote $\widehat{\mathcal{S}}^{(j)}$ as the reduced system in the $j$-th iteration. More precisely, we approximate $\eta^{(j)}$ by a "tail" expression $\tau^{(j)}$ and use the stopping criterion $\frac{\abs{\tau^{(j+1)}-\tau^{(j)}}}{\abs{\tau^{(1)}}}<10^{-6}$ (see \cite{FasGP26} for more details).

Moreover, we rescale the input $u$ according to \cite{benner_lyapunov_2011,condon_nonlinear_2005} for ensuring the existence of the Gramian $P$; see \cite{damm_direct_2008,FasGP25}. The rescaling 
 $u\to\frac{1}{\gamma}u$, $0<\gamma<1$, also enhances the convergence of the fixed-point iterations in \Cref{alg:T_BQO_TSIA} and \cite[Algorithm 6.1]{FasGP26}. We use $\gamma=0.1$.

The examples and algorithms are implemented in MATLAB where we use the solvers from the MESS library to solve Sylvester equations and the standard and generalized Lyapunov equations. The methods employed are a modified version of
\texttt{mess\_sylvester\_sparse\_dense} with default settings,
\texttt{mess\_lyap} with default settings, and  \texttt{mess\_lyapunov\_bilinear} with the options  
 residual norm tolerance 
\texttt{opts.blyap.res\_tol} $ =10^{-8}$, 
relative difference tolerance 
\texttt{opts.blyap.rel\_diff\_tol} $ =10^{-7}$, and 
 maximum iterations 
\texttt{opts.blyap.maxiter} $ =50$ \cite{Mess}.
The code for the numerical examples is available on ZENODO (\url{https://doi.org/10.5281/zenodo.23692641}).
Computations were run on an Intel(R) Core(TM) i7-1255U CPU @ 1.70 GHz with 16 GB RAM using MATLAB 2025b.

\begin{figure}[h]
\centering
    \includegraphics[width=0.6\textwidth]{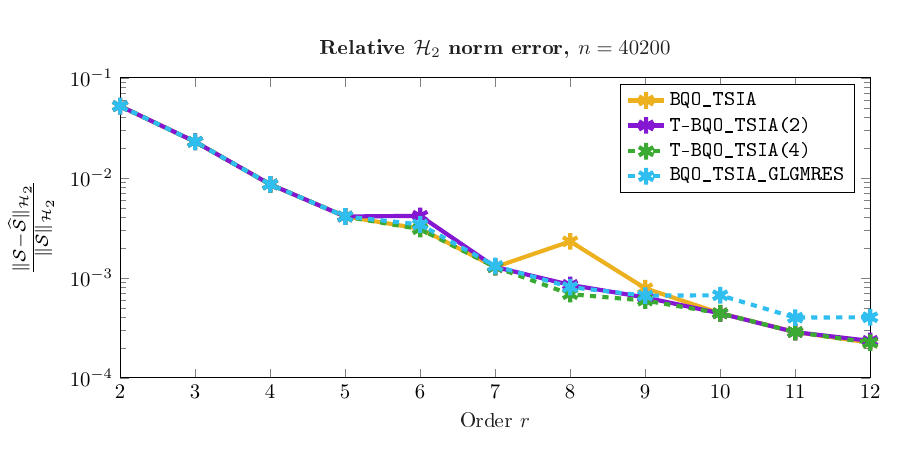}
    \caption{RC example, Relative $\sH_2$ norm errors.}
    \label{fig:1}
\end{figure}

First, we study a nonlinear RC circuit example, originally from \cite{bai_projection_2006,benner_lyapunov_2011}, which is bilinearized via Carleman bilinearization \cite{rugh_nonlinear_1981}. In \cite{FasGP26}, we chose an alternative output to build a BQO system from this example. For more details, see \cite{FasGP26}. 
\Cref{fig:1} shows the relative $\sH_2$ norm error between the original and the reduced BQO systems for different reduced orders. 
Here, we observe that the studied methods produce similar results for smaller orders but yield small differences for higher orders. More precisely, the method \texttt{BQO\_TSIA\_GLGMRES} provides slightly worse results here. Also there is a convergence issue for $r=8$, where \texttt{BQO\_TSIA} does not converge to a global optimal model. This problem could be tackled by a different initialization of \Cref{alg:T_BQO_TSIA}; see \cite{FasGP26}. Nevertheless, we retain our initialization chosen in \cite{FasGP26}.
More importantly, we observe that the truncated methods $\texttt{T-BQO\_TSIA(K)}$ match the accuracy of the non-polynomial methods, at a much lower computational cost. 

\begin{table}[htbp]
\footnotesize
\caption{RC example, Timings in sec}\label{tab:1}
\begin{center}
  \begin{tabular}{c|cccccc} 
    Reduced order $r$ & 2 & 4 & 6 & 8 & 10 & 12 \\
    \hline 
    \texttt{BQO\_TSIA} & 11.68 & 20.34 & 87.13 & 60.62 & 62.26 & 57.49 \\
    \texttt{T-BQO\_TSIA(2)} & 4.85 &8.18 &14.77 &23.47&22.68&22.09\\
    \texttt{T-BQO\_TSIA(4)} &9.56&15.72&126.33 &63.11&42.74&40.04 \\
    \texttt{BQO\_TSIA\_GLGMRES}&0.52&1.01&10.73 &4.07&35.99&4.67 
  \end{tabular}
\end{center}
\end{table}

Next, we take a look at the timings in \Cref{tab:1} for the RC example and observe that the truncated variants offer a speedup for the original \texttt{BQO\_TSIA} in most cases. Nevertheless, the method \texttt{BQO\_TSIA\_GLGMRES} is still faster than the truncated variants \texttt{T-BQO\_TSIA(K)}, $K=2,4$, except for $r=10$. In this case, there is again a convergence issue which can be improved by a different choice of an initial reduced BQO system.

\begin{figure}[h]
\centering
    \includegraphics[width=0.62\textwidth]{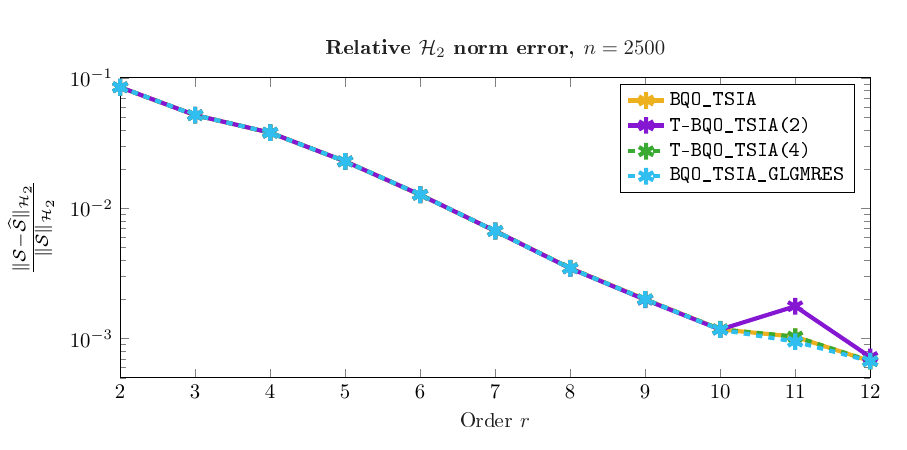}
    \caption{Heat example, Relative $\sH_2$ norm errors.}
    \label{fig:2}
\end{figure}

As a second example, we test the proposed method on a heat equation example from \cite{benner_lyapunov_2011} where we modified the output and applied decaying inputs, according to \cite{FasGP26}, to construct a BQO system. In \Cref{fig:2}, we display the relative $\mathcal{H}_2$ norm error for increasing reduced orders. For this example, all methods yield very similar results. This shows the good approximation quality of the truncated algorithms while being computationally more efficient than the original methods.
There is only a visible difference for $r=11$ where \texttt{T-BQO\_TSIA(2)} cannot compete with the other variants and yields a slightly worse error. As we see for \texttt{T-BQO\_TSIA(4)}, the increase of $K$ in \Cref{alg:T_BQO_TSIA} can lead to better results, but also leads to longer computation times. We can observe this in \Cref{tab:2} where the timings for the methods applied to the heat example are presented. The truncated variants of \texttt{BQO\_TSIA}, namely \texttt{T-BQO\_TSIA(2)} and \texttt{T-BQO\_TSIA(4)}, offer a speedup compared to the full \texttt{BQO\_TSIA}. The method \texttt{T-BQO\_TSIA(2)} is even faster than \texttt{BQO\_TSIA\_GLGMRES}.

\begin{table}[htbp]
\footnotesize
\caption{Heat example, Timings in sec}\label{tab:2}
\begin{center}
  \begin{tabular}{c|cccccc} 
    Reduced order $r$ & 2 & 4 & 6 & 8 & 10 & 12 \\
    \hline 
    \texttt{BQO\_TSIA} & 0.9 & 0.77 & 1.76 & 1.4 & 2.65 & 2.29 \\
    \texttt{T-BQO\_TSIA(2)} & 0.31 &0.28 &0.59 &0.54&0.83&0.71\\
    \texttt{T-BQO\_TSIA(4)} &0.64&0.51&1.12 &0.89&2.74&1.46 \\
    \texttt{BQO\_TSIA\_GLGMRES}&0.53&0.4&0.81 &0.67&4.5&1.27 
  \end{tabular}
\end{center}
\end{table}

\section{Conclusion}
\color{black}
We developed an $\sH_2$ optimal model reduction algorithm for approximating polynomial BQO systems. We introduced an $\sH_2$ norm and  derived an $\sH_2$ error expression for these systems which allowed us to derive the first-order optimality conditions. We showed that these truncated $\sH_2$  optimality conditions are consistent with the original full $\sH_2$ optimality conditions. Based on this theoretical analysis, we formulated a truncated $\sH_2$ optimal algorithm and tested it on two numerical examples. We observed that our truncated algorithm is a competing alternative to the full $\mathcal{H}_2$ optimal algorithm. While requiring to solve only regular Sylvester equations, the proposed algorithm matches the accuracy of the untruncated version even for small truncation indices.

\section{Disclosure statement} 

We wish to confirm that there are no known conflicts of interest associated with this publication and there has been no significant financial support for this work that could have influenced its outcome. Part of this work was completed during a research stay of the third author at Virginia Tech supported by a fellowship of the German Academic Exchange Service (DAAD). The work of Gugercin was supported in part by the National Science Foundation (NSF), United States under Grant No. CMMI-2130695.

We confirm that the manuscript has been read and approved by all named authors and that there are no other persons who satisfied the criteria for authorship but are not listed. We further confirm that the order of authors listed in the manuscript has been approved by all of us.

We confirm that we have given due consideration to the protection of intellectual property associated with this work and that there are no impediments to publication, including the timing of publication, with respect to intellectual property. In so doing we confirm that we have followed the regulations of our institutions concerning intellectual property.

We understand that the Corresponding Author is the sole contact for the Editorial process. He is responsible for communicating with the other authors about progress, submissions of revisions and final approval of proofs.

\printcredits

\bibliographystyle{model1-num-names}
% \bibliographystyle{cas-model2-names}

\bibliography{bib_H2bqo}

\end{document}